\documentclass[12pt]{article}
\usepackage[utf8]{inputenc}

\usepackage{amsfonts,amsmath}
\usepackage{hyperref}
\usepackage{xcolor}
\usepackage{titlesec}
\usepackage{amsmath,amsthm}
\usepackage{bm}
\usepackage{stackengine}
\usepackage{tikz}
\usepackage{graphicx}
\usepackage{geometry}

\newcommand\xrowht[2][0]{\addstackgap[.5\dimexpr#2\relax]{\vphantom{#1}}}

\newtheorem{theorem}{Theorem}[section]
\newtheorem{proposition}[theorem]{Proposition}
\newtheorem{corollary}[theorem]{Corollary}

\newtheorem{lemma}[theorem]{Lemma}
\theoremstyle{definition}
\newtheorem{definition}[theorem]{Definition}
\newtheorem{remark}[theorem]{Remark}
\newtheorem{example}[theorem]{Example}

\definecolor{cornellred}{rgb}{0.7, 0.11, 0.11}

\newcommand{\B}{\mathfrak{B}}
\newcommand{\N}{\mathbb{N}}
\newcommand{\Z}{\mathbb{Z}}
\newcommand{\R}{\mathbb{R}}
\renewcommand{\P}{\mathbb{P}}

\newcommand{\w}{\omega}

\providecommand{\keywords}[1]{\textit{{Keywords.}} #1}

\providecommand{\MSC}[1]{\textit{{2010 MSC.}} #1}

\makeatletter
\def\smallunderbrace#1{\mathop{\vtop{\m@th\ialign{##\crcr
   $\hfil\displaystyle{#1}\hfil$\crcr
   \noalign{\kern3\p@\nointerlineskip}
   \tiny\upbracefill\crcr\noalign{\kern1\p@}}}}\limits}
\makeatother

\titleformat{\section}
{\normalfont\large\scshape}{\thesection}{1em}{}
\titleformat{\subsection}
{\normalfont\normalsize\scshape}{\thesubsection}{1em}{}
\titleformat{\subsubsection}
{\normalfont\normalsize\bfseries}{\thesubsubsection}{1em}{}
\titleformat{\paragraph}[runin]
{\normalfont\normalsize\scshape}{\theparagraph}{1em}{}
\titleformat{\subparagraph}[runin]
{\normalfont\normalsize\bfseries}{\thesubparagraph}{1em}{}
\usepackage{libertine}

\title{\textsc{Normal and pseudonormal numbers}}

\author{\textsc{Nicolò Cangiotti}\footnote{Dipartimento di Matematica, Politecnico di Milano, via Bonardi 9, 20133 Milano, Italy. E-mail Address: \texttt{nicolo.cangiotti@polimi.it}.} \ \& \textsc{Daniele Taufer}\footnote{CISPA Helmholtz Center for Information Security, 66123 Saarbrücken, Germany. E-mail Address: \texttt{daniele.taufer@cispa.de}.}}
\date{}

\begin{document}

\maketitle
\begin{abstract}
\noindent
After a short review of the historical milestones on normal numbers, we introduce the Borel numbers as the reals admitting a probability function on their different bases representations.
In this setting, we provide two probabilistic characterizations of normality based on the stochastic independence of their digits.
Finally, we define the pseudonormality condition, which is satisfied by normal numbers and may be evaluated in a finite number of steps.
\bigskip

\noindent
\smallskip

\noindent
\MSC{\emph{Primary:} 11K16, 11A63. \emph{Secondary:} 60A05.}
\end{abstract}

\section{Introduction}
A real number is called \emph{normal} if every finite sequence of digits is uniformly distributed in its base-$b$ representation, for all positive integers $b$.
This concept was first introduced in 1909 by Émile Borel \cite{Borel1909}, who proved that almost all real numbers are normal. 
He actually proved that the set of non-normal numbers has Lebesgue measure zero by means of the Borel–Cantelli lemma, but such a non-constructive approach led many mathematicians to better explore the nature of these numbers.

Despite more than a century has passed since the first definition of these numbers, there are still many open questions and conjectures to be addressed, among which verifying the normality condition of renowned candidates and constructing concrete examples of normal numbers seem to be, quite surprisingly, the most compelling.
In 1916, Wac\l{}aw Sierpiński constructed a famous example of a number that can be proved to be normal \cite{Sierp1917}.
In the following years, other normal numbers were introduced \cite{Champernowne33,Nakai1992} and many of their properties were discovered \cite{Schmidt1960}.
One of the most interesting and useful tools to deal with these objects is the so-called \emph{Weyl criterion} \cite{Weyl1916}, which has led to several results (see, e.g., \cite{Wall49}) and seems promising also for future research.
However, we are still far from a comprehensive knowledge of these objects.

The present work aims at providing a different approach to the study of normal numbers by using alternative techniques. 
First, we rigorously formalize the probabilistic concept underlying the definition of normality. Initially, we restrict our attention to the numbers that admit a probability function on all their bases representations, which we call \emph{Borel numbers} as they were originally considered in \cite{Borel1909}.
Among those, normal numbers are then characterized as the real numbers whose base representations constitute strings of mutually independent digits. However, both this probabilistic approach and the classical ones share a common limitation: for a given base $b$ and a given precision Borel number $\w \in \R$, the number of tests that should be performed for verifying the base-$b$ normality of $\w$ is unbounded.
\bigskip

The paper is organized as follows.
In Section \ref{Sec2} we recall the definition of normal number and the most known and relevant results in this field.
In Section \ref{Sec3} we define the Borel numbers and then we propose two probabilistic characterizations of normality.
Section \ref{Sec4} is devoted to the study of the class of pseudonormal numbers.
Finally, in Section \ref{Sec5}, we take the stock of our work, proposing some possible future developments.

\section{Basic definitions and known results}
\label{Sec2}
This section is devoted to recall the classic notation about normal numbers and review some of their key properties.
Throughout this work, we will always denote the set of positive integers by $\N^* = \N \setminus \{0\}$, and by $\N_{>k}$ the set of integers greater than a given value $k$. For a more exhaustive collection of known results and properties about normal numbers, we refer to \cite{Quef2006}.

\subsection{Definitions and examples}
We begin with the basic definitions of normality proposed in the first work by Borel \cite{Borel1909}.

\begin{definition}
A real number $\w$ is said to be \emph{simply normal} in an integer base $b$ if its infinite sequence of digits is \emph{uniformly distributed}, i.e., every base-$b$ digit has natural density $1/b$ in $\w_{b}$.
We let $\mathcal{S}_b$ denote the set of simply normal numbers in the given base $b$.
\end{definition}

\begin{definition}
\label{Normalb}
A real number $\w$ is said to be \emph{normal in base b} (or \emph{b-normal}) if, for every positive integer $n$, all possible $n$-digits strings have density $b^{-n}$ in the base-$b$ representation of $\w$. The set of all $b$-normal numbers will be indicated by $\mathcal{N}_b$.
\end{definition}

\begin{definition}
\label{absnorm}
A real number $\w$ is called an \emph{(absolutely) normal number} if it is normal in base $b$ for every $b \in \N^*$.
\end{definition}

\begin{remark} \label{rmk:limExistance}
We say that an $n$-digits string $L$ has natural density $\epsilon \in \R$ in the $b$-representation of $\w$ if the limit
\begin{equation*}
    \lim_{m \to \infty} \frac{|\{\text{substrings equal to}\ L \ \text{in the first}\ m \ \text{digits of}\ \w \}|}{m}
\end{equation*}
exists, and it is equal to $\epsilon$.
 
We just remark that the existence of such limit is not guaranteed. As an example, the number
\begin{equation*}
    \w = \sum_{i = 0}^\infty \left(\sum_{j = 0}^{2^{2i}-1} b^{-j}\right) b^{-2^{2i}}
\end{equation*}
has the following base-$b$ representation
\begin{equation*}
    w_b = 0,\smallunderbrace{1}_{1}\underbrace{00}_{2}\underbrace{1111}_{4}\underbrace{00000000}_{8}\underbrace{1111111111111111}_{16}\dots .
\end{equation*}
It is easy to verify that the above limit for the null sequence $L=[00 \dots 0]$ of any length does not exist.
\end{remark}

\begin{example}
There are many classical examples of $b$-normal numbers. One of the most famous is the so-called Champernowne constant:
\[
C_{10}=0,12345678910111213141 \dots
\]
which was proved to be a $10$-normal number \cite{Champernowne33}.
Nakai and Shiokawa proved \cite{Nakai1992} that $C_b$ is normal for any base $b$, hence for instance
\begin{align*}
C_{2} &= 0,11011100101110111100 \dots\\
C_{3} &= 0,12101112202122100101 \dots
\end{align*}
are $2$-normal and $3$-normal numbers, respectively.
\end{example}
\begin{remark}
In 1917 Sierpinski performed a labored construction to define an absolutely normal number, however it was a priori unclear whether his number was computable \cite{Sierp1917}.
In fact, in 2002, Becher and Figueira proved that Sierpinski's number is \emph{computably enumerable}, alongside providing a recursive reformulation of its construction, which produces an actual computable absolutely normal number \cite{Becher2002}.
\end{remark}

It is clear that a rational number cannot be normal, as its sequences of digits are eventually periodic in any base.
Moreover, Martin proved in 2001 the existence of an absolutely abnormal irrational number \cite{Martin01}, namely a number that is not simply normal in any base.

\begin{example}
\label{ExAbnormal}
The absolutely abnormal number $\mu$ due to Martin can be computed as follows. Let
\begin{equation*}
   f(n)=\begin{cases}
   4 &\text{if} \ n=2, \\
   n^{\frac{f(n-1)}{n-1}} &\text{if} \ n \ge 3,
    \end{cases}
\end{equation*}
then we define Martin's number as
\begin{equation*}
\mu=\prod_{m=2}^{\infty} \left ( 1- \frac{1}{f(m)} \right),
\end{equation*}
which turns out to be irrational and absolutely abnormal.
\end{example}

\subsection{Classical results}
The importance of normal numbers is witnessed by the following fundamental theorem proved by Borel \cite{Borel1909}.

\begin{theorem} \label{thm:Borel}
Almost all real numbers are normal.
\end{theorem}

We recall that two given integers $r,s \in \N^*$ are called \emph{multiplicatively dependent} if there exist integers $n, m \in \mathbb{Z}$ such that $r^n = s^m$.
In such case, we let $r \sim s$ denote this relation. Multiplicatively dependent bases have been proved to be normal-equivalent by Schmidt in \cite{Schmidt1960}, who obtained the following result.

\begin{theorem}
Let $r,s \in \N^*$.
\begin{itemize}
    \item If $r \sim s$, then any number normal to base $r$ is also normal to base $s$.
    \item If $r \not\sim s$, the set of numbers that are normal to base $r$ but not even simply normal to base $s$ has the power of the continuum.
\end{itemize}
\end{theorem}

It seems now appropriate to point out the connection between normal numbers and uniformly distributed modulo $1$ sequences (see, e.g, {\cite[Thm. 8.15]{Niven1956}}), which has been constituting one of the most practical results in this field.

\begin{theorem}
\label{Thm-unif1}
Let $b \in \N^*$. A real number $\w$ is $b$-normal if and only if the sequence $\{b^n\w\}_{n \in \N}$ $=$ $\{\w, b \w, b^2\w, \dots\}$ is uniformly distributed modulo $1$.
\end{theorem}

One of the most powerful tools to exploit this result is the following \emph{criterion} proved by H. Weyl \cite{Weyl1916}.

\begin{theorem}
\label{Thm-unif2}
Let $\{a_n\}_{n \in \mathbb{N}}$ be a sequence of real numbers.
The following are equivalent.
\begin{itemize}
    \item $\{a_n\}_{n \in \mathbb{N}}$ is uniformly distributed modulo $1$.
    \item For each non-zero integer $k$, we have
    \begin{equation*}
        \lim_{N \to \infty} \frac{1}{N} \sum_{j = 1}^N e^{2 \pi i k a_j} = 0.
    \end{equation*}
   
\end{itemize}
\end{theorem}

\begin{corollary}
The sequence $\{n\w\}_{n\in \N^*}=\{\w, 2\w, 3\w, \dots\}$ is equidistributed modulo $1$ if and only if $\w$ is irrational.
\end{corollary}

\begin{remark}
In literature, there are many results concerning the equidistribution modulo $1$, such as \cite{Bern1911,Niven1956,Ruzsa1982} .
\end{remark}
    
\section{Probabilistic characterization of normality}
\label{Sec3}

First and foremost, we fix the notation that will be employed in the following.
For any $b \in \N^*$, we denote the base-$b$ digits by $\B_b = \{0, \dots, b-1\}$, or simply by $\B$ when the base is clear from the context.
Given a real number $\w \in \R$, we denote the string of digits in its $b$-expansion by $\w_b$, indexed by the inverse of correspondent exponent of $b$ in its $b$-representation, namely
\begin{equation*}
    \left( \sum_{i = k}^\infty s_i b^{-i} \right)_b = [s_k, s_{k+1}, \dots].
\end{equation*}

We are primarily interested in numbers whose $b$-expansions admit the limits discussed in Remark \ref{rmk:limExistance}.

\begin{definition}
We say that $\w \in \R$ is a \emph{Borel number} if for every $b,n \in \N^*$ and every string $[s_1,\dots,s_n] \in \B_b^n$, the following limit
\begin{equation*}
    \lim_{m \to \infty} \frac{|\{[b_k,\dots,b_{k+n-1}] \subset \w_b \ : \ [b_k,\dots,b_{k+n-1}] = [s_1,\dots,s_n] \}_{ 1 \leq k \leq m-n+1}|}{m}
\end{equation*}
exists. When existent, we denote it by $\P_{\w,b}([s_1,\dots, s_n])$, or in a more concise form by $\P([s_1,$ $\dots, s_n])$, when the number and the base are clear from the context. We denote the set of all Borel numbers by $\mathcal{B}$.
\end{definition}

Dealing with Borel numbers, we will simply let $[s_1,\dots, s_n]$ denote the event that a generic $n$-substring of digits in the $b$-representation of $\w$ is equal to $[s_1, \dots, s_n]$.
In the same spirit, we denote the probability of finding the digit $s_n$ right after an occurrence of the string $[s_1, \dots, s_{n-1}]$ in the $b$-representation of $\w$ by $\P([s_n] \ | \ [s_1,\dots, s_{n-1}])$.
Finally, we let $\P([s_n] \ | \ [s_1])$ denote the probability of finding in the $b$-representation of $\w$ a string of length $n$ starting with $s_1$ and ending with $s_n$.

\begin{remark}
Since every normal number is a Borel number by definition, Theorem \ref{thm:Borel} ensures that almost all real numbers are Borel.
\end{remark}

The idea of normality is closely connected to the stochastic independence of the digits and this section shall formalize this link. 

\begin{definition}
\label{caucon}
Let $b \in \N^*$.
A Borel number $\w \in \R$ is said $b$-\emph{casually constructed} if for every $n \in \N^*$ and every $[s_1, \dots, s_n]$ $\in \B^n$ the following conditions hold.
\begin{equation}
\label{Hyp0}
\begin{cases}
    \P([s_1]) = \frac{1}{b}, \\
    \P([s_{n}] \ | \ [s_1,\dots, s_{n-1}]) = \frac{1}{b}.
    \end{cases}
\end{equation}

If the condition \eqref{Hyp0} holds for every base $b \in \N^*$, we shall say that $\w$ is \emph{casually constructed}.
\end{definition}

\begin{remark}
Definition \ref{caucon} means that any digit is independent of the preceding.
In fact, two events $A$ and $B$ are stochastically independent if
\begin{equation*}
    \P(A|B)=\P(A).
\end{equation*}
\end{remark}

The equivalence between $b$-casually constructed numbers and $b$-normal numbers is given by the following proposition.

\begin{proposition}
Let $b \in \N^*$. A Borel number $\w \in \R$ is $b$-normal if and only if it is $b$-casually constructed.
\end{proposition}
\begin{proof}
Let us first assume that $\w$ is $b$-normal.
Since $\w$ is simply $b$-normal, then for every $s_1 \in \B$ we have   $\P([s_1]) = \frac{1}{b}$.
Furthermore, by definitions of conditional probability and normality, we have
\begin{gather*}
\scalebox{1}{$\displaystyle{\P([s_{n}] \ | \ [s_1,\dots, s_{n-1}]) = \frac{\P([s_{n}]\wedge [s_1,\dots, s_{n-1}])}{\P( [s_1,\dots, s_{n-1}])} = \frac{\P([s_1,\dots, s_{n}])}{\P([s_1,\dots, s_{n-1}])}= \frac{b^{-n}}{b^{-(n-1)}}=\frac{1}{b}}$},
\end{gather*}
hence both the conditions \eqref{Hyp0} are satisfied.
\smallskip

Let us now suppose that $\w$ is $b$-casually constructed.
We inductively prove that $\w$ is $b$-normal, i.e., we show by induction on $n \in\N^*$ that any string of size $n$ has the same probability $b^{-n}$ to occur.
The base step is the first condition of \eqref{Hyp0}.
As for the inductive step, we have
\begin{align*}
\P([s_1,\dots, s_{n+1}]) = \P([s_1,\dots, s_{n}])\cdot \P([s_{n+1}] \ | \ [s_1 ,\dots, s_n]) = \frac{1}{b^n} \cdot \frac{1}{b} =\frac{1}{b^{n+1}}.
\end{align*}
Thus, $\w$ is $b$-normal.
\end{proof}

The next result provides a further characterization from a probabilistic point of view, which turns out to be very useful in the following. 

\begin{theorem}
\label{Thm:Indep}
Let $b \in \N^*$ and let $\w$ be a Borel number. Then, $\w$ is $b$-normal if and only if for every $n \in \N^*$ and every pair $s_1, s_n \in \B$, we have
\begin{equation}
\label{Hyp1b}
    \P([s_{n}] \ | \ [s_1]) = \frac{1}{b}.
\end{equation}
\end{theorem}
\begin{proof}
First, we assume that $\w$ is $b$-normal. Thus, all the strings of length $n$ are equiprobable, and $1/b^2$ of them start with $s_1$ and end with $s_n$. Therefore
\begin{equation*}
    \frac{1}{b^2} = \P([s_1] \wedge [s_n]) = \P([s_n] \ | \ [s_1]) \P([s_1]),
\end{equation*}
which proves condition \eqref{Hyp1b} as $\P([s_1]) = \frac{1}{b}$.
\smallskip

Now let us suppose that hypothesis \eqref{Hyp1b} holds. We notice that 
\begin{equation*}
    \P([s_n]) = \sum_{s_1 \in \B} \P([s_n] | [s_1]) \P([s_1]) = \frac{1}{b} \sum_{s_1 \in \B} \P([s_1]) = \frac{1}{b}.
\end{equation*}
Therefore, by hypothesis \eqref{Hyp1b} we get for every $1 \leq i,j \leq n$:
\[
\P([s_i] \ | \ [s_j])=\P([s_i]).
\]
Thus, the $[s_i]$ are mutually independent, which implies
\[
\P([s_1,\dots, s_n])=\P([s_1])\cdot \dots \cdot \P([s_n]) = \frac{1}{b^n},
\]
hence $\w$ is normal.
\end{proof}

\begin{remark}
In other words, Theorem \ref{Thm:Indep} states that a Borel number is normal if and only if its digits are mutually independent.
\end{remark}

\section{The \texorpdfstring{$\Delta_b$}--scheme and pseudonormal numbers}
\label{Sec4}

Given a base $b \in \N^*$ and a real number $\w = \sum_{i = k}^\infty s_ib^{-i} \in \R$, we let 
\begin{equation}
\label{Dh}
    D_h = \{i \in \Z_{\geq k} \ :\ s_i = h\}
\end{equation}
denote the positions of the digits $h$ in the base-$b$ representation of $\w$.
Thus, we have a convenient representation of $\w$ as
\begin{equation}
\label{w_newrep}
    w = \sum_{h \in \B} \sum_{t \in D_h} h \cdot b^{-t}.
\end{equation}

\begin{definition}
Let $b \in \N^*$ and $i,j \in \B_b$. We define the \emph{swap map}
\begin{equation}
\tag{Swap map}
\sigma_{(b)}^{i,j}:\mathbb{R} \to \mathbb{R},
\end{equation}
which swaps the digits $i$ and $j$ in the base-$b$ representation of the input number, i.e., in the above notation it swaps $D_i$ and $D_j$.
We will simply refer to it as $\sigma^{i,j}$ when the base is clear from the context.
\end{definition}
\begin{example}
By considering the first twenty $10$-digits of $\pi$, we have
\begin{align*}
    \pi_{} &= 3.14159265358979323846 \dots \\
    \sigma^{1,5}_{(10)}\left (\pi_{} \right ) &=3.\bm{5}4\bm{5}\bm{1}926\bm{1}3\bm{1}8979323846 \dots
    \end{align*}
\end{example}

\begin{definition}
For every $b \in \N^*$ and $i,j \in \B_b$ we define
\[
\Delta_{b}^{i,j}(\w) = \w_{(b)}-\sigma_{(b)}^{i,j}(\w).
\]
Again, we will omit the base when it is clear from the context.
\end{definition}

We observe that, regardless from $b$, there are at most six different digits in the base-$b$ representation of  $\Delta_{b}^{i,j}(\w)$, and they only depend on $|i-j|$.

\begin{lemma}
\label{lemma:Digits}
Let $b \in \N^*$ and $\w \in \R$. For every $i,j \in \B$ let $\delta = |i-j|$. The digits appearing in $\Delta_b^{i,j}(\w)$ are contained in the set
\begin{equation*}
\mathfrak{D} = \{\delta,b-\delta,\delta-1,b-\delta-1,b-1,0\},
\end{equation*}
where all the quantities in $\mathfrak{D}$ has to be intended modulo $b$.
\end{lemma}
\begin{proof}
First, we notice that the the $D_h$'s defined in \eqref{Dh} are disjoint sets and, by making use of the notation \eqref{w_newrep}, we have
\begin{equation*}
    \Delta_b^{i,j}(\w) = \sum_{t \in D_i} (i-j) \cdot b^{-t} + \sum_{t \in D_j} (j-i) \cdot b^{-t}.
\end{equation*}
As considering the opposite of a number does not change its digits in any base, we may assume without loss of generality that $i\geq j$, so that $\delta = i-j$.
We now write

\begin{equation*}
    - \delta \cdot b^{-t} = -b^{-t+1}+(b-\delta) \cdot b^t,
\end{equation*}
so that
\begin{equation*}
    \Delta_b^{i,j}(\w) = \sum_{t \in D_i} \delta \cdot b^{-t} + \sum_{t \in D_j} (b - \delta) \cdot b^{-t} - \sum_{t \in D_j} b^{-t+1}.
\end{equation*}
Since $D_i$ and $D_j$ are disjoint, the base-$b$ digits that may appear in the above term $\sum_{t \in D_i} \delta \cdot b^{-t} + \sum_{t \in D_j} (b - \delta) \cdot b^{-t}$ are only $\delta, b-\delta$ and $0$.
By subtracting $\sum_{t \in D_j} b^{-t+1}$ to it, we may only reduce some of these digits by $1$ modulo $b$, i.e., we may also find $\delta-1, b-\delta-1$ and $b-1$.
\end{proof}

\begin{definition}
\label{DeltaScheme}
We define the \emph{$\Delta_b$-scheme} of a Borel number $\w$ as the table listing, for every $i,j \in \B_b$, the probability of the $b$-digits in $\Delta_{b}^{i,j}(\w)$.
\end{definition}

\begin{remark}
When $\sigma_b^{i,j}(\w) = \w$, namely when $i = j$ or when neither $i$ nor $j$ appear in $\w$, clearly correspond to lines of zeros in the $\Delta_b$-scheme, and may be omitted.
Analogously, since $\sigma^{i,j}_b(\w)=\sigma^{j,i}_b(\w)$, we can compute it only once.
\end{remark}

\begin{example}
Let us consider the $5$-representation of $\w=19/62$, which is given by $\w_5=0.\overline{123}$. Its $\Delta_{5}$-scheme is the following table.

\begin{center}
\scalebox{1.0}{
\begin{tabular}{||c c|| c c c c c ||} 
\hline\xrowht{10pt}
 $i$ & $j$ & $\P([0])$ & $\P([1])$ & $\P([2])$ & $\P([3])$ & $\P([4])$\\ [.5ex] 
\hline\hline \xrowht{10pt}
0 & 1 & 0.$\overline{6}$ & 0.$\overline{3}$ & 0  & 0 & 0\\
 \hline \xrowht{10pt}
 0 & 2 & 0.$\overline{6}$  & 0 & 0.$\overline{3}$ & 0 & 0\\
 \hline \xrowht{10pt}
 0 & 3 & 0.$\overline{6}$  & 0 & 0 & 0.$\overline{3}$ & 0\\
 \hline \xrowht{10pt}
 1 & 2 & 0.$\overline{6}$ & 0 & 0 & 0 & 0.$\overline{3}$\\
 \hline \xrowht{10pt}
 1 & 3 & 0 & 0.$\overline{3}$ & 0 & 0.$\overline{3}$ & 0.$\overline{3}$\\
 \hline \xrowht{10pt}
 1 & 4 & 0.$\overline{6}$ & 0 & 0 & 0.$\overline{3} $ & 0\\
 \hline \xrowht{10pt}
 2 & 3 & 0.$\overline{6}$ & 0 & 0 & 0 & 0.$\overline{3}$\\
 \hline \xrowht{10pt}
 2 & 4 & 0.$\overline{6}$ & 0 & 0.$\overline{3} $ & 0 & 0\\
 \hline \xrowht{10pt}
 3 & 4 & 0.$\overline{6}$ & 0.$\overline{3}$ & 0 & 0 & 0\\
 \hline
\end{tabular}
}
\end{center}
\end{example}

Some numbers have a $\Delta_b$-scheme with a particularly simple description, as in the following definition.

\begin{definition}
\label{DefInduct}
Let $b \in \N^*$. A Borel number $\w \in \R$ is $b$\emph{-pseudonormal} if the probability of the $b$-digits appearing in $\Delta_{b}^{i,j}(\w)$ only depends on $\delta = |i-j|$ and are given by
\begin{center}
 \begin{tabular}{|| c c c c c c ||} 
 \hline \xrowht{10pt}
 $\P([\delta])$ & $\P([b-\delta])$ & $\P([\delta-1])$ & $\P([b-\delta-1])$ & $\P([b-1])$ & $\P([0])$ \\ [0.5ex] 
  \hline\hline \xrowht{15pt}

 $\frac{1}{2b}$ & $\frac{1}{2b}$ & $\frac{1}{2b}$ & $\frac{1}{2b}$ & $\frac{b-2}{2b}$ & $\frac{b-2}{2b}$ \\
 [0.2ex]\hline
 \end{tabular}
 \end{center}
 where the probability of the same digits are intended to be summed. The set of all $b$-pseudonormal numbers is denoted by $\mathcal{D}_b$.
\end{definition}

\begin{example}
Let $b=5$ and $\delta=1$, which give $\mathfrak{D}=\{0,1,3,4\}$.
We have
\begin{align*}
    b-\delta=b-1=4 \implies \P([b-\delta])+\P([b-1])=\frac{b-1}{2b}=\frac{2}{5},
\end{align*}
and similarly
\begin{align*}
    \delta-1=0 \implies \P([0])+\P([\delta-1])=\frac{b-1}{2b}=\frac{2}{5}.
\end{align*}
Thus, we obtain the first line of the $\Delta_5$-scheme: 
\begin{center}
 \begin{tabular}{||c|| c c c c c   ||} 
 \hline \xrowht{10pt}
 $\delta$ & $\P([0])$ & $\P([1])$ & $\P([2])$ & $\P([3])$ & $\P([4])$ \\ [0.5ex] 
  \hline\hline \xrowht{10pt}

 $1$ & $0.40$ & $0.10$ & $0$ & $0.10$ & $0.40$  \\
\hline \xrowht{10pt}
 $2$ & $0.30$ & $0.10$ & $0.20$ & $0.10$ & $0.30$  \\
\hline \xrowht{10pt}
 $3$ & $0.30$ & $0.10$ & $0.20$ & $0.10$ & $0.30$  \\
 \hline \xrowht{10pt}
 $4$ & $0.40$ & $0.10$ & $0$ & $0.10$ & $0.40$  \\
 [0.2ex]\hline
 \end{tabular}
 \end{center}
\end{example}

\begin{remark}
From Definition \ref{DefInduct} we notice that it is sufficient to consider the $\Delta_b$-scheme of any $b$-pseudonormal number $\w$ by only evaluating $\Delta_b^{0,\delta}(\w)$ for $\delta$ ranging into $\left \{0,\dots, \lfloor \frac{b}{2} \rfloor \right \}$. In fact, modulo $b$, we have:
\begin{equation*}
    |\delta|=|-\delta|=|b-\delta|.
\end{equation*}
\end{remark}

The following proposition motivates the name \emph{pseudonormal}, by showing that those numbers have the same $\Delta_b$-scheme of any $b$-normal number.

\begin{proposition}
\label{PropD}
Let $b \in \N^*$ and let $\w \in \R$ be a Borel number. If $\w$ is $b$-normal then it is $b$-pseudonormal.
\end{proposition}
\begin{proof}
Let us represent $\w$ as in \eqref{w_newrep}. As for the proof of Lemma \ref{lemma:Digits}, for every $i \neq j$ and $\delta = |i-j|$, we have
\begin{equation*}
    \Delta_b^{i,j}(\w) = \sum_{t \in D_i} \delta \cdot b^{-t} + \sum_{t \in D_j} (b - \delta) \cdot b^{-t} - \sum_{t \in D_j} b^{-t+1}.
\end{equation*}
Let $d \in D_i \cup D_j$ be a position index and $s(d)$ be the corresponding digit, i.e.,
\begin{equation*}
    s(d) = \begin{cases}
        \delta &\text{if } d \in D_i,\\
        b-\delta &\text{if } d \in D_j,
    \end{cases}
\end{equation*}
and let us also call $n_d$ its following index in $D_i \cup D_j$, namely
\begin{equation*}
    n_d = \inf_{n \in D_i \cup D_j}\{n > d\}.
\end{equation*}
We may re-write $\Delta_b^{i,j}(\w)$ as
\begin{equation*}
    \Delta_b^{i,j}(\w) = \sum_{\substack{d \in D_i \cup D_j\\ n_d \in D_j}} \big( s(d) \cdot  b^{n_d-d-1} - 1  \big) \cdot b^{-n_d+1} + \sum_{\substack{d \in D_i \cup D_j\\ n_d \not\in D_j}} s(d) \cdot b^{-d}.
\end{equation*}
The convenience of the above expression lies in the fact that every digit of $\Delta_{b}^{i,j}(\w)$ is determined by either the first or the second series above, and they do not overlap. Let $\w$ be $b$-normal, by Theorem \ref{Thm:Indep} we have
\begin{equation*}
\label{Pnd}
    \P(n_d \in D_i \ | \ d \in D_i \cup D_j) = \frac{1}{b}. 
\end{equation*}
Similarly, the same theorem combined with the Bayes theorem provides
\begin{equation*}
    \P\big(s(d) = \delta \ | \ n_d \in D_i\big) = \P\big(s(d) = \delta \ | \ n_d \in D_j\big) = \frac{1}{2}.
\end{equation*}
Therefore, the probabilities of the digits $\delta, b-\delta, \delta-1$ and $b-\delta-1$ are all equal to $\frac{1}{2b}$.
Finally, since the average length of $n_d-d$ is $\frac{b}{2}$ regardless of $n_d$, then the probabilities of the digits $0$ and $b-1$ are equal, therefore they are both equal to
\begin{equation*}
    \frac{1}{2}\left(1 - 4\left(\frac{1}{2b}\right)\right) = \frac{b-2}{2b}.
\end{equation*}
Thus, the probability of the digits of $\Delta_b^{i,j}(\w)$ is that of Definition \ref{DefInduct}.
\end{proof}

\begin{example}
\label{Ex:base4}
Let us consider the number $\w \in \R$ defined by its base-$4$ representation
\begin{equation*}
  \w_4=0.\overline{33323130232221201312111003020100}.
\end{equation*}

The $\Delta_4$-scheme of $\w$ is given by the table below. 

\begin{center}
\scalebox{1.0}{
 \begin{tabular}{||c c|| c c c c ||} 
 \hline\xrowht{10pt}
 $i$ & j & $\P([0])$ & $\P([1])$ & $\P([2])$ & $\P([3])$ \\ [0.5ex] 
 \hline\hline\xrowht{10pt}
 0 & 1 & 0.50000 & 0.12500 & 0.12500  & 0.25000 \\
 \hline\xrowht{10pt}
 0 & 2 & 0.25000 & 0.25000 & 0.25000 & 0.25000 \\
 \hline\xrowht{10pt}
 0 & 3 & 0.25000 & 0.12500 & 0.12500 & 0.50000 \\
 \hline\xrowht{10pt}
 1 & 2 & 0.53125 & 0.09375 & 0.09375 & 0.28125\\
 \hline\xrowht{10pt}
 1 & 3 & 0.21875 & 0.25000 & 0.25000 & 0.28125 \\
 \hline\xrowht{10pt}
 2 & 3 & 0.53125 & 0.09375 & 0.09375 & 0.28125 \\
 \hline 
\end{tabular}
}
\end{center}
In particular, we notice that the scheme does not depend only on $\delta$ as the lines indexed by $(i,j)=(0,1)$ and $(i,j)=(1,2)$ are different.
As consequence, $\w$ does not satisfy Definition \ref{DefInduct} so it is not $4$-pseudonormal. Thus, by Proposition \ref{PropD}, $\w$ is not $4$-normal as well, which is not surprising since it is rational.  
\end{example}

We highlight that Proposition \ref{PropD} does not provide a characterization of normality, as shown by the following example. 

\begin{example}
\label{DeltaNotNormal}
Let us consider  the following real number $\w$ in its base-$4$ representation:
\begin{equation*}
    \w_4=\overline{1320201331020231}
\end{equation*}
It may be checked that $\w_4$ is $4$-pseudonormal, but it is rational so it cannot clearly be normal.
\end{example}

The following theorem summarizes, together with Theorem \ref{Thm:Indep}, the main results of our work.
\begin{theorem}
\label{finalthm}
Let $b \in \N^*$ and let $\w \in \R$ be a Borel number. The following chain of implications holds:
\begin{gather*}
w \text{ is } b \text{-causally constructed} \Leftrightarrow \w \text{ is } b \text{-normal} \Rightarrow \\
\Rightarrow \w \text{ is pseudonormal } \Delta_b \text{-scheme} \Rightarrow \w \text{ is simply } b \text{-normal}.
\end{gather*}
\end{theorem}

\begin{remark}
Testing $b$-normality for a number of a given precision should require a countably infinite number of tests, both by using the original definition (\emph{every} length should be checked) and by joint use of Theorem \ref{Thm-unif1} and Theorem \ref{Thm-unif2} (\emph{every} parameter $k\in \Z^*$ should be checked).
Conversely, testing $b$-pseudonormality is a task that may be performed in a finite number of steps, i.e., those needed to construct the corresponding $\Delta_b$-scheme.
As pseudonormality is a necessary condition for normality, it is easily envisioned as a preliminary, computationally feasible, test on the candidate numbers.
\end{remark}

\begin{example}
By considering the first $10^8$ digits of $\pi$, we computationally check that the digits of $\Delta_{10}^{i,j}(\pi)$ have the following (approximated) distribution scheme.

\begin{center}
\scalebox{0.85}{
 \begin{tabular}{||c|| c c c c c c c c c c ||}

 \hline \xrowht{10pt}
$\delta$ & $\P([0])$ & $\P([1])$ & $\P([2])$ & $\P([3])$ & $\P([4])$ & $\P([5])$ & $\P([6])$ & $\P([7])$ & $\P([8])$ & $\P([9])$ \\ [0.5ex]
 \hline\hline \xrowht{10pt}
 1 & 0.45 & 0.05 & 0 & 0 & 0 & 0 & 0 & 0 & 0.05 & 0.45 \\
 \hline \xrowht{10pt}
 2 & 0.40 & 0.05 & 0.05 & 0 & 0 & 0 & 0 & 0.05 & 0.05 & 0.40 \\
 \hline \xrowht{10pt}
 3 & 0.40 & 0 & 0.05 & 0.05 & 0 & 0 & 0.05 & 0.05 & 0 & 0.40 \\
 \hline \xrowht{10pt}
 4 & 0.40 & 0 & 0 & 0.05 & 0.05 & 0.05 & 0.05 & 0 & 0 & 0.40 \\
 \hline \xrowht{10pt}
 5 & 0.40 & 0 & 0 & 0 & 0.10 & 0.10 & 0 & 0 & 0 & 0.40 \\
 \hline
\end{tabular}
}
\end{center}
This scheme hints at the fact that $\pi$ might well be $10$-normal, as it is widely believed.
The very same distribution holds for $\log(2)$, $\sqrt{2}$ and the gold ratio $\varphi$, approximated with the same precision.
Quite surprisingly, it holds also for the Fibonacci constant
\begin{equation*}
    F_{10} = 0.1123581321345589144233\dots.
\end{equation*}
\end{example}

\section{Conclusions}
\label{Sec5}
The history of mathematics is littered with many unsolved problems and open questions and it is certain that, even today, issues around normal numbers are among the most curious and interesting. This work aims at giving a new boost to the study of this topic, by combining different points of view.
The final recap provided by Theorem \ref{finalthm} is portrayed by the following chain of inclusions.
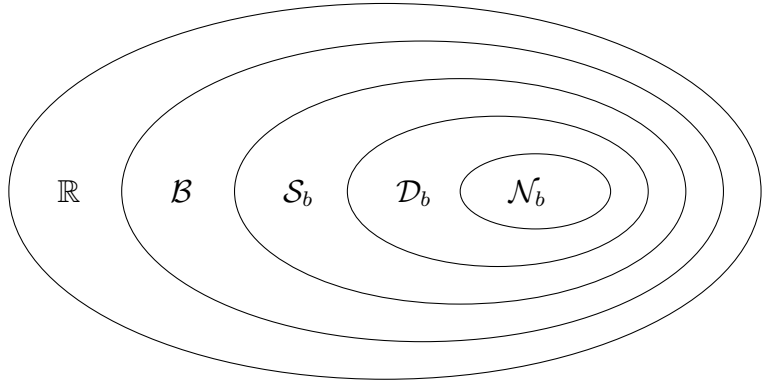
\begin{center}
 \begin{tikzpicture}
      \draw (0.5,0) ellipse (3cm and 1.5cm);
      \draw (1,0) ellipse (2cm and 1cm);
      \draw (1.5,0) ellipse (1cm and 0.5cm);
      \draw (0,0) ellipse (4cm and 2cm);
    \draw (-0.5,0) ellipse (5cm and 2.5cm);
    \draw
    (1,0.3) node[anchor=north west] {$\mathcal{N}_b$}
    (-0.5,0.3) node[anchor=north west] {$\mathcal{D}_b$}
    (-2,0.3) node[anchor=north west] {$\mathcal{S}_b$}
    (-3.5,0.3) node[anchor=north west] {$\mathcal{B}$}
    (-5,0.3) node[anchor=north west] {$\mathbb{R}$};
    \end{tikzpicture}
\end{center}
The above inclusions are all strict, as it follows from the examples shown by Remark \ref{rmk:limExistance}, Example \ref{Ex:base4} and Example \ref{DeltaNotNormal}.
 
Although this probabilistic approach is intuitively clear and easy to handle, there are still many directions that demand to be better explored.

First, even though almost all numbers are proved to be Borel, no efficient criteria are known to decide if it is the case. However, even negative results in this direction might be advantageous for constructing examples of provably non-normal irrational numbers.

Moreover, a similar criterion for normality or pseudonormality may be envisioned, possibly with better performance or scope than the one that is presented in the current work.

Finally, pseudonormal non-normal numbers such as that of Example \ref{DeltaNotNormal} are rare and practically difficult to construct. This suggests that there may be some other light conditions that, paired with pseudonormality, are actually equivalent to normality.


\par\bigskip\noindent
\textsc{Acknowledgments.}
NC and DT would like to thank the University of Pavia and the Helmholtz Center for Information Security, respectively, for supporting their research.

\noindent
This article has been also supported in part by the European Union’s H2020 Programme under grant agreement number ERC-669891.

\bibliographystyle{plain}
\bibliography{references}

\begin{thebibliography}{10}

\bibitem{Becher2002}
V.~Becher and S.~Figueira.
\newblock An example of a computable absolutely normal number.
\newblock {\em Theoretical Computer Science}, 270:947--958, 01 2002.

\bibitem{Bern1911}
F.~Bernstein.
\newblock \"uber eine anwendung der mengenlehre auf ein aus der theorie der
  s\"akularen störungen herr\"uhrendes problem.
\newblock {\em Mathematische Annalen}, 71:417--439, 1911.

\bibitem{Borel1909}
{\'{E}}.~Borel.
\newblock Les probabilit{\'{e}}s d{\'{e}}nombrables et leurs applications
  arithmétiques.
\newblock {\em Rendiconti del Circolo Matematico di Palermo}, 27(1):247--271,
  1909.

\bibitem{Champernowne33}
D.~G. Champernowne.
\newblock The construction of decimals normal in the scale of ten.
\newblock {\em Journal of the London Mathematical Society}, s1-8(4):254--260,
  1933.

\bibitem{Martin01}
G.~Martin.
\newblock Absolutely abnormal numbers.
\newblock {\em The American Mathematical Monthly}, 108(8):746--754, 2001.

\bibitem{Nakai1992}
Y.~Nakai and I.~Shiokawa.
\newblock Discrepancy estimates for a class of normal numbers.
\newblock {\em Acta Arithmetica}, 62(3):271--284, 1992.

\bibitem{Niven1956}
I.~Niven.
\newblock {\em Irrational Numbers}.
\newblock Carus Mathematical Monographs. Mathematical Association of America,
  1956.

\bibitem{Quef2006}
M.~Quefféec.
\newblock {\em Old and new results on normality}, volume~48 of {\em Lecture
  Notes--Monograph Series}, pages 225--236.
\newblock Institute of Mathematical Statistics, Beachwood, Ohio, USA, 2006.

\bibitem{Ruzsa1982}
I.~Z. Ruzsa.
\newblock On the uniform and almost uniform distribution of $(a_n x)$ mod. 1.
\newblock {\em Seminaire de Théorie des Nombres de Bordeaux}, 12:1--22,
  1982-1983.

\bibitem{Schmidt1960}
W.~M. Schmidt.
\newblock On normal numbers.
\newblock {\em Pacific Journal of Mathematics}, 10(2):661--672, 1960.

\bibitem{Sierp1917}
W.~Sierpi{\'n}ski.
\newblock D\'emonstration \'el\'ementaire du th\'eor\`eme de m. borel sur les
  nombres absolument normaux et d\'etermination effective d'une tel nombre.
\newblock {\em Bulletin de la Soci\'et\'e Math\'ematique de France},
  45:125--132, 1917.

\bibitem{Wall49}
D.~D. Wall.
\newblock {\em Normal Numbers}.
\newblock PhD thesis, Berkeley, California: University of California, 1949.

\bibitem{Weyl1916}
H.~Weyl.
\newblock Über die gleichverteilung von zahlen mod. eins.
\newblock {\em Mathematische Annalen}, 77:313--352, 1916.

\end{thebibliography}

\end{document}